\documentclass[reqno]{amsart}
\usepackage{xcolor,fullpage,hyperref,cleveref,enumitem}
\usepackage{tikz,amsmath,amssymb,color,mathtools}
\usetikzlibrary{calc}
\usepackage{float}
\usepackage{pgf}
\usepackage{amsrefs}
\usepackage[normalem]{ulem}
\usepackage{listings}
\usepackage{multirow}

\theoremstyle{definition} 
\newtheorem{theorem}{Theorem}[section]

\newtheorem{proposition}[theorem]{Proposition}

 \newtheorem{example}{Example}
 
\newtheorem{lemma}[theorem]{Lemma}
\newtheorem{claim}{Claim}

\theoremstyle{definition}
\newtheorem{definition}{Definition}[section]
\theoremstyle{remark}
\newtheorem*{remark}{Remark}

\newcommand{\NN}{\mathbb{N}}

\newcommand{\out}{\mathcal{O}}

\newcommand{\Sym}{\mathfrak{S}}
\DeclareMathOperator\pk{pk}
\DeclareMathOperator\Pk{Pk}
\DeclareMathOperator\Av{Av}

\DeclareMathOperator\red{red}

\newcommand{\Bn}[3]{B_{#1}^{(#2)}(#3)}

\usepackage{thmtools} 
\usepackage{thm-restate}

\title{A parking function analog of the Schr\"{o}der numbers}

\author[Martinez]{Lucy Martinez}
\address[L.~Martinez \& D.~Zeilberger]{Department of Mathematics, Rutgers University, Piscataway, NJ 08854}
\email{\textcolor{blue}{\href{mailto:lucy.martinez@rutgers.edu}{lucy.martinez@rutgers.edu}}}

\author[Zeilberger]{Doron Zeilberger}
\email{\textcolor{blue}{\href{mailto:doronzeil@gmail.com}{doronzeil@gmail.com}}}

\begin{document}

\begin{abstract}
Pattern avoidance is a central topic in the study of permutations. Parking functions provide a natural setting in which to ask analogous questions. In this paper, we study parking functions whose parking permutations avoid the patterns $1234$ and $2134$, obtaining an analog of the Schr\"{o}der numbers for parking functions. To enumerate these objects, we develop an enumeration scheme that runs in polynomial time and adapt the prefix-based framework introduced by Zeilberger in 1998, and later improved by Vatter, by instead using ``suffixes.''
\end{abstract}

\maketitle

\section{Introduction}
Enumerating pattern avoiding permutation classes has a long history. In 1985, Simion and Schmidt enumerated permutations avoiding a pair of patterns of length $3$. Many further enumerative results have followed for patterns in permutations and in words, see B\'{o}na~\cite{MBona12}, Kitaev~\cite{Kitaev} and Vatter~\cite{Vatter15} for surveys. Numerous systematic techniques have been developed for enumerating permutation classes.
These include generating trees~\cite{CGHK, Vatter06, West95, West96}, insertion encoding~\cite{insertionencoding}, substitution decompositions~\cite{subdecomp}, and enumeration schemes~\cite{Zeilberger98}. 

In 1998, Doron Zeilberger introduced prefix enumeration schemes to count pattern avoiding permutations by providing an automated method for counting Wilf equivalence classes~\cite{Zeilberger98}. However, the success rate for finding these schemes was limited. In 2005, Vatter extended Zeilberger's results by introducing ``gap vectors'', which completely automated the enumeration of many more Wilf classes making the enumeration scheme much more powerful~\cite{Vatter08}. Later, Zeilberger reformulated Vatter's schemes back to his original notation, allowing for even quicker enumeration of many pattern avoiding permutation classes~\cite{Zeilberger06}.
Other work on enumeration schemes includes that of Pudwell, who studied pattern avoidance in words and barred permutations by automating new enumeration results~\cite{Pudwell, barred}. More recently, a new (systematic) algorithmic framework was introduced called combinatorial exploration, which automatically and rigorously analyzes the structure of combinatorial objects and derives their counting sequences and generating functions~\cite{CEAF}.

A superset of permutations is the set of parking functions. Parking functions can be described using a deterministic parking process. Consider a parking lot with $n$ parking spots on a one-way street labeled $1$ to $n$. A sequence of $n$ cars enters the street one by one, from left to right, with car $i$ having a preferred spot $a_i$, which we call its parking preference. For each $i\in [n]$, car $i$ drives to its preferred spot $a_i$ and attempts to park. If the spot is not available, the car continues to drive down the street and parks in the next available spot, if one exists. If there is no available spot, the car exits the lot and is unable to park. We call this the parking rule, and the sequence $\alpha$ is called a \emph{parking function} if all cars can park under the parking rule. Konheim and Weiss established that there are $(n+1)^{n-1}$ parking functions of length $n$~\cite{konheim1966occupancy}. For a parking function $\alpha$, the order in which the cars park on the street is called the \emph{parking permutation} and is denoted by $\out(\alpha)$. If $\Sym_n$ denotes the set of permutations of $[n]$ written in one-line notation, then the parking permutation of $\alpha$ is 
\[\out(\alpha)=\pi_1\pi_2\cdots\pi_n,\]
where $\pi_i$ denotes that car $\pi_i$ parked in spot $i$ on the street. 

Since there is a many-to-one correspondence between parking functions and permutations, it is natural to study pattern avoidance in parking functions. Qiu and Remmel introduced a notion of pattern avoidance for parking functions by viewing them as labeled Dyck paths~\cite{QiuRemmel}. Adeniran and Pudwell~\cite{adeniranpudwell} further studied this notion and posed several open problems, which were subsequently resolved by Adenbaum~\cite{Adenbaum26}. Yan introduced a second notion of pattern avoidance in parking functions by studying the parking permutations~\cite{junyan}. These parking permutations record the order in which the cars park on the street. Yan enumerated all parking functions whose parking permutations avoid any collection of patterns of length $3$. In this paper, we use Yan's notion of pattern avoidance. 

The main point of the present article is to extend the Zeilberger/Vatter enumeration scheme for permutations to the setting of parking functions.
We adapt the enumeration scheme, which is based on prefixes for permutations, to instead use suffixes of parking permutations. As a proof of concept, we apply this framework to the illustrative case of $\{1234,2134\}$-avoiding parking permutations, systematically enumerate the corresponding parking functions, and obtain a recurrence that can be computed in polynomial time. In the Wilf sense, the corresponding permutation class is enumerated by the Schr\"{o}der numbers~\cite{Kremer}. More broadly, we hope that this approach can eventually be fully automated, as in the original setting of permutation classes, although the current framework still requires human assistance to identify and prove the necessary structural properties.
All data produced and used in our computations are collected on a public website available from~\cite{Maple}, where we also provide the three Maple packages accompanying this article.

This paper is organized as follows. In Section~\ref{sec:background}, we introduce the necessary background on parking functions and pattern avoidance. In Section~\ref{sec:scheme}, we reformulate the prefix enumeration scheme in terms of suffixes and extend it to parking enumeration schemes. We illustrate this approach by recovering Yan's enumeration~\cite[Theorem 2.28]{junyan} of parking functions whose parking permutations are $123$-avoiding. In Section~\ref{sec:mainresult}, we apply the resulting framework to the patterns $1234$ and $2134$, and we obtain the enumeration of parking functions whose parking permutations avoid these patterns.

\section{Background and Notation}\label{sec:background}

\subsection{Parking Functions}
Colmenarejo et al.~\cite{knaple} provide a technique for enumerating parking functions by counting through permutations. We restate their result below and use Lemma~\ref{lem:numberofprefsforeachcar} to construct sequences (Definition~\ref{def:pvector}), which we refer to as parking tables. These sequences play a role in the main result.

\begin{lemma}\label{lem:numberofprefsforeachcar}
Given a permutation $\pi=\pi_1\cdots\pi_n\in\Sym_n$, define
\[\ell(\pi_i)=\max\{k: \, \pi_j \leq \pi_i \text{ for all } i-k+1\leq j\leq i\}.\]
The number of parking functions with parking permutation $\pi$ is 
\[\prod_{i=1}^n \ell(\pi_i).\]
\end{lemma}

\begin{example}\label{ex: count from 23154}
Let $n=5$ and $\pi=23154$ be the permutation that parks the cars on the street. Then, $\ell(\pi_1)=1, \ell(\pi_2)=2, \ell(\pi_3)=1, \ell(\pi_4)=4$ and $\ell(\pi_5)=1$. Moreover, the product of $1\cdot 2 \cdot 1 \cdot 4 \cdot 1=8$ is the total number of parking functions with parking permutation $\pi=23154$.
\end{example}

The following definition constructs sequences from Lemma~\ref{lem:numberofprefsforeachcar}.
\begin{definition}\label{def:pvector}
Let $\pi=\pi_1\pi_2\cdots\pi_n\in\Sym_n$. For all $i\in[n]$, define the sequence $p=(p_1,p_2,\ldots,p_n)\in [n]^n$ by
\[p_i=\ell(\pi_i)=\max\{k: \, \pi_j \leq \pi_i \text{ for all } i-k+1\leq j\leq i\}.\] 
We call $p$ the \emph{parking table} and often write $p$ in one-line notation.
\end{definition}

\subsection{Pattern Avoidance}\label{subsec:patternavoid}
For $m\leq n$ and $\sigma \in \Sym_m, \pi\in \Sym_n$, we say that $\pi$ \emph{contains} $\sigma$ as a pattern if there exists $1\leq i_1< \cdots <i_m\leq n$, such that $\pi(i_a)<\pi(i_b)$ if and only if $\sigma(a)<\sigma(b)$ for all $a,b\in [m]$, and we say $\pi$ \emph{avoids} $\sigma$ otherwise. For any collection $\sigma_1,\ldots, \sigma_k$ of permutations, $\Av_n(\sigma_1,\ldots,\sigma_k)$ denotes the set of all permutations in $\Sym_n$ containing none of the permutations $\sigma_1,\ldots,\sigma_k$ as a pattern. For two permutations $\sigma, \tau \in \Sym_k$, we say that $\sigma$ and $\tau$ are \emph{Wilf equivalent} if $|\Av_n(\sigma)|=|\Av_n(\tau)|$ for all $n\geq k$.

An equivalent characterization of pattern containment is given in terms of ``reduction'' or ``standardization'' of a word. This perspective is particularly useful in our later enumeration arguments for the parking analogs of the Schr\"{o}der numbers.

\begin{definition}\label{def:reduction of perm}
Let $[k]^n$ denote the set of words of length $n$ in the alphabet $\{1,2,\ldots,k\}$. Let $w\in [k]^n$ and $w=w_1w_2\cdots w_n$. The reduction (or standardization) of $w$, denoted $\red(w)$, is the unique word of length $n$ obtained by replacing the $i^{th}$ smallest distinct letter of $w$ with $i$, for each $i$. We refer to $\red(w)$ as the \emph{reduced permutation}.
\end{definition}

\begin{example}\label{ex:reductionex1}
If $\pi=23154$, deleting the entry $\pi_3=1$ yields the word $2354$, whose reduced permutation is $\red(2354)=1243$.
\end{example}

The following equivalent definition of pattern containment will be convenient.

\begin{definition}
Let $\pi=\pi_1\pi_2\cdots\pi_n\in \Sym_n$ and let $\sigma=\sigma_1\sigma_2\cdots\sigma_k\in \Sym_k$. We say that $\pi$ contains $\sigma$ if there exist $1\leq i_1<i_2<\cdots <i_k\leq n$ so that $\red(\pi_{i_1}\pi_{i_2}\cdots\pi_{i_k})=\sigma$. Otherwise $\pi$ avoids $\sigma$.
\end{definition}

\begin{claim}\label{claim:reduavoid}
Let $k\leq n \in \NN$ and $\Pi$ be a collection of patterns. If $\pi\in \Av_n(\Pi)$ then $\red(w)\in \Av_{k}(\Pi)$ for any subsequence $w$ of $\pi$ of length $k$.
\end{claim}

Notice that the parking table sequences can be defined for words with distinct values. If $w$ is a word with distinct values and parking table $p$, then its reduced permutation also yields the same parking table $p$. 

\begin{claim}\label{lem:parkingtables}
Let $w\in[k]^n$ with distinct elements and $\sigma\in \Sym_n$. If $\sigma=\red(w)$ for a word $w$ with distinct elements then the parking tables of $w$ and $\sigma$ coincide.
\end{claim}

\begin{example}[Continuing Example~\ref{ex:reductionex1}]
The parking table of the word $2354$ is $p=(1,2,3,1)$ which is the same parking table of its reduced permutation $1243$.
\end{example}

\begin{definition}\label{def:suffix}
Let $\pi=\pi_1\cdots \pi_n\in \Sym_n$ be a parking permutation and $p=p_1p_2\cdots p_n\in[n]^n$ be its parking table. For $1\leq  k \leq j \leq n$, we call a pair $(\tau, p')$ a \emph{bi-suffix} of $\pi$ if \[\tau=\pi_k\pi_{k+1}\cdots \pi_j \in[n]^{j-k+1},\] and \[p'=p_kp_{k+1}\cdots p_j \in [n]^{j-k+1}.\]
\end{definition}

\begin{example}\label{ex:suffix1}
Let $\pi=2471635 \in \Sym_7$ and $k=4,j=7$. The parking table of $\pi$ is given by $p=1231212$. Then the bi-suffix of $\pi$ is the tuple $(\tau,p)$ where $\tau=1635$ and $p=1212$.
\end{example}

\begin{lemma}
Given $\pi=\pi_1\cdots\pi_n\in \Sym_n$ with bi-suffix $(i,r)$, where $\pi_n=i$ and $\ell(\pi_n=i)=r$. Let $w\in[n]^{n-1}$ be the word obtained by deleting $\pi_n=i$ and $\red(w)=\tau=\tau_1\cdots\tau_{n-1}\in \Sym_{n-1}$.
Then
\[ \prod_{i=1}^n \ell(\pi_i)=r\prod_{i=1}^{n-1}\ell(\tau_i).\]
In other words, the number of parking functions with parking permutation $\pi$ is the same as the product of $r$ and the number of parking functions with parking permutation $\tau$.
\end{lemma}
\begin{proof}
This follows by applying Lemma~\ref{lem:parkingtables} to $w=\pi_1\pi_2\cdots\pi_{n-1}$ and $\red(w)=\tau$. Since the parking tables of both $w$ and $\tau$ are the same, then 
\[\prod_{i=1}^{n-1}\ell(\tau_i)=\prod_{i=1}^{n-1} w_i= \prod_{i=1}^{n-1} \ell(\pi_i).\]
Thus, \[ \prod_{i=1}^n \ell(\pi_i)=\ell(\pi_n)\prod_{i=1}^{n-1}\ell(\tau_i)=r\prod_{i=1}^{n-1}\ell(\tau_i).\]
\end{proof}

\begin{definition}
Let 
$\Pi=\{\sigma_1,\sigma_2,\ldots, \sigma_k\}$ be collection of patterns. For $m\geq 1$, let $B_n^{(m)}(\Pi)$ denote the set of all bi-suffixes $(\tau, p')$, with $\tau, p'\in [n]^m$, that arise from permutations $\pi \in \Av_n(\Pi)$ and its corresponding parking table $p\in[n]^n$. That is,  
\[ 
\Bn{n}{m}{\Pi}=\{(\tau, p'): \, (\tau, p') \text{ is a bi-suffix of } \pi \in \Av_n(\Pi) \text{ with parking table } p \}.
\]
We refer to $\tau$ and $p'$ as the permutation component and parking table component, respectively.

For each $(\tau,p')\in \Bn{n}{m}{\Pi}$, we define
\[
\Bn{n}{m+1}{\Pi;\, (\tau,p')}
=
\left\{
(\mu,q):
\begin{array}{l}
(\mu,q)\text{ is a bi-suffix of }\pi\in\Av_n(\Pi)
\text{ with parking table }p,\\
\mu_2\cdots\mu_{m+1}=\tau
\text{ and }
q_2\cdots q_{m+1}=p'
\end{array}
\right\}.
\]

\textit{Warning:} We abuse notation whenever $\tau=i$ and $p'=r$ are length $1$ and simply write $(i,r)$ instead of $((i), (r))$.
\end{definition}

\begin{example}
Let $n=4$ and $\pi=2413\in \Av_4(123)$ with parking table $p=1212$. Then, $(3,2)$ is a member in the set $\Bn{4}{1}{123}$ while $((1,3), (1,2))$ is a member in the set $\Bn{4}{2}{123;\, (3,2)}$.
\end{example}

\section{From Permutation Enumeration Schemes to Parking Enumeration Schemes}\label{sec:scheme}
We review the scheme first described by Zeilberger~\cite{Zeilberger98} and later improved by Vatter~\cite{Vatter08}. We do not go into the details of the enumeration scheme, but we follow the notation used in Zeilberger's survey~\cite{Zeilberger06}. 
The goal of an enumeration scheme is to describe a class of avoiding permutations in terms of smaller permutations in the class via recursion. The original scheme method used prefixes. We implement Vatter's Maple package so that we can use suffixes, which is available in~\cite{Maple}. Using suffixes allows us to analyze the last elements of the permutation along with the last elements of the corresponding parking table.

We begin with an example of the enumeration scheme for the $123$-avoiding class and illustrate how to adapt the enumeration scheme to parking functions by recovering Yan's result~\cite[Theorem 2.28]{junyan}. Although we do not include the details here, we also recover the enumeration of parking functions whose parking permutations are $\{123,132\}$-avoiding and of those whose parking permutations are $\{123,231\}$-avoiding (see \texttt{PWilfCases.txt} in~\cite{Maple}).
 
\subsection{The $123$-avoiding permutation scheme}\label{subsec:123perm}
Let $A(n)$ be the set of $123$-avoiding permutations, i.e. $A(n)\coloneq \Av_n(123)$.
We can describe the set $A(n)$ by taking collections of the permutations fixed by their last element. Precisely, define
\[A_1(n,i)\coloneq \{\pi \in A(n): \, \pi_n=i\}. \]
Hence, \[ A(n)= \bigcup_{i=1}^n A_1(n,i),\]
which is the union of all permutations that end with the element $i$ for all $i\in[n]$.

Moreover, the last two elements of any $123$-avoiding permutation form a word in the alphabet $[n]$ of length $2$. If $n\geq 2$, then we can define the following two sets according to the reduction of the last two elements,
\[A_{12}(n,j,i)\coloneq \{\pi\in A(n): \, \pi_{n-1}=j, \pi_{n}=i, j<i\}, \]
and
\[A_{21}(n,h,i)\coloneq \{\pi\in A(n): \, \pi_{n-1}=h, \pi_n=i, h>i\}. \]
Here, $A_{12}(n,j,i)$ is the set of all permutations whose last two elements end with $ji$ and $j<i$. Similarly, $A_{21}(n,h,i)$ is the set of all permutations whose last two elements end with $hi$ and $h>i$. 

In the enumeration scheme, the goal is to keep describing sets in terms of collections of smaller sets until the enumeration is equivalent for some collections. If $i=1$ then $A_1(n,i)$ is equinumerous to the set $A(n-1)$ since if $\pi\in A_1(n,i)$ then removing $i$ from $\pi$ yields a word $w$ of length $n-1$ whose reduction $\red(w)$ is a permutation of length $n-1$. The same argument follows if $i=2$. In both cases, both reductions avoid the pattern $123$~(Claim~\ref{claim:reduavoid}).

If $i\geq 3$, then we have
\[A_1(n,i)=\left(\bigcup_{j=1}^{i-1} A_{12}(n,j,i)\right)\cup \left(\bigcup_{h=i+1}^{n} A_{21}(n,h,i)\right).\]

We show that $|A_{21}(n,h,i)|=|A_1(n-1,h-1)|$. Let $\pi \in A_{21}(n,h,i)$ then removing the element $i$ from $\pi$ results in the word $w=\pi_1\pi_2\cdots\pi_{n-2}h$ of length $n-1$, and $\red(w)$ replaces every element $k$ in $w$ by $k-1$ whenever $k>i$. Thus, the reduced permutation ends with $h-1$ because $k=h$ and $h>i$. In addition, the reduced permutation avoids the pattern $123$~(Claim~\ref{claim:reduavoid}).
Thus, the set $A_{21}(n,h,i)$ is equinumerous with the set $A_1(n-1,h-1)$ for $i+1\leq h \leq n$, or $|A_{21}(n,h,i) |= |A_1(n-1,h)|$ with $i\leq h \leq n-1$.

For the set $A_{12}(n,j,i)$, the suffix scheme relies on \emph{gap vectors}. Gap vectors give a condition for which choices of elements in the suffix yield empty permutations. Precisely, if $\pi=\pi_1\pi_2\cdots\pi_j i_1i_2i_3\cdots i_{n-j}$ then a gap vector provides information on the possible values of the elements $i_1,i_2,\ldots, i_{n-j}$. In our example, we obtain the information that $A_{12}(n,j,i)\neq \emptyset $ if and only if $j=1$. Hence, if $j=1$ and $\pi\in A_{12}(n,j,i)$, then removing the element $j=1$ from $\pi$ results in the word $w=\pi_1\pi_2\cdots i$ of length $n-1$, and $\red(w)$ replaces every element $k$ in $w$ by $k-1$ whenever $k>j$. Thus, the reduced permutation ends with $i-1$ because $k=i$ and $i>j=1$. In addition, the reduced permutation avoids the pattern $123$~(Claim~\ref{claim:reduavoid}). Thus, the set $ A_{12}(n,j,i)$ is equinumerous to $A_1(n-1,i-1)$.

Therefore, if $i\geq 3$, then
\[|A_1(n,i)|=\sum_{h=1}^{n-1} |A_1(n-1,h)| + |A_1(n-1,i-1)|, \]
otherwise $i=1$ or $i=2$ then $|A_1(n,i)|=|A(n-1)|$. 
Since the enumeration scheme for the $123$-avoiding permutations only checked up to the last two elements in the permutation, we say that the enumeration scheme has depth $2$. 
\begin{proposition}\label{prop:123recurrence}
Let $a(n)$ be the number of $123$-avoiding permutations and $a(n,i)$ be those permutations that end with the element $i$. Then
\[a(n,i)=
\begin{cases}
    a(n-1) & \text{if } i=1 \text{ or } i=2 \\
    a(n-1,i-1)+\sum_{j=i}^{n-1} a(n-1,j) & \text{otherwise}
\end{cases}
\]
with initial conditions $a(1,1)=1$ and $a(1,i)=0$ if $i\geq 2$.
\end{proposition}

\subsection{The $123$-avoiding parking scheme}
The goal of the parking enumeration scheme is to determine how the recurrence of the permutation scheme changes with the parking tables. A natural question is whether there exists constants coefficients, depending on the values $n$ and $i$, for the recurrence in Proposition~\ref{prop:123recurrence}.

We begin with a characterization of the suffixes for the $123$-avoiding permutations. 
\begin{lemma}
The set of bi-suffixes in $\Bn{n}{1}{123}$ is given by
\[
\Bn{n}{1}{123}=\{(i,r): \, 1\leq r \leq i \leq n-1\} \cup \{(n,n)\}.
\]
In other words, if a $123$-avoiding parking permutation ends in $i<n$, then the corresponding parking table can end in any $r\in [i]$, while if it ends in $n$, the parking table necessarily ends in $n$.
\end{lemma}
\begin{proof}
For any permutation $\pi\in \Av_n(123)$, the last element $\pi_n=i$ can range from any value from $1$ to $n$.

If $\pi_n=n$ then all elements to the left of $\pi_n$ are less than $n$ so $\ell(\pi_n)=p_n=n$ in the parking table. Hence, $(n,n)\in \Bn{n}{1}{123}$. 
If $\pi_n=i$ and $i<n$, by definition, if $\ell(\pi_n)=r$ then there are $r-1$ elements to the left of $\pi_n$ that are all at most $i$. Since there are only $i$ elements of $[n]$ that are at most $i$, we have $r\leq i$. Thus, $(i,r)\in \Bn{n}{1}{123}$ for $1\leq r\leq i \leq n-1$.

We now assume $ 1\leq r\leq i\leq n-1$ and construct a $123$-avoiding permutation whose last element $\pi_n=i$ and whose last entry in the parking table $p_n=r$.
Consider 
\[\pi= (i-1)(i-2)\cdots r n (n-1)\cdots (i+1)(r-1)(r-2)\cdots 1 i.\]
Notice that $\pi$ can be decomposed into four blocks. The first block $B_1$ contains the elements $r, r+1, \ldots , i-2, i-1$. If $i=r$, then $B_1$ is an empty block, otherwise, if $r<i$ then there are $i-r$ elements in block $B_1$ in decreasing order.
The second block $B_2$ contains the elements $i+1,i+2,\ldots, n-1, n$ in decreasing order, and there are $n-i$ of them. The third block $B_3$ contains the elements $1,2,\ldots, r-2,r-1$ in decreasing order and there are $r-1$ of them. The last block $B_4$ is the last element $i$. Altogether, there are $n$ elements in the permutation $\pi$.

Furthermore, $\ell(\pi_n)=r$ because the first element $j$ to the left of $\pi_n=i$ such that $j>i$ is at position $n-r$. Mainly, $\pi_{n-r}=i+1$ and $\pi_{n-r}>\pi_n=i$. For any $n-r+1\leq k \leq n-1 $, $\pi_k \leq r-1 < i$ and so $p_n=\ell(\pi_n)=r$.
Moreover, $\pi$ avoids the pattern $123$. Indeed, elements in blocks $B_1$, $B_2$ and $B_3$ are decreasing. Also, for any $b_1\in B_1, b_2\in B_2$ and $b_3\in B_3$, $b_3<b_1<b_2$ and $b_1<i<b_2$. Thus, there are no increasing subsequences as $b_3<b_1<i<b_2$.
\end{proof}

\begin{example}\label{ex:bisuffixlength1}
For $n=3$, there are five $123$-avoiding permutations, and for each, we list the bi-suffixes below.
\begin{table}[H]
    \centering
    \begin{tabular}{c|c|c}
       Permutation $\pi$ & Parking table $p$ & $(i,r)$ where $i=\pi_3$ and $r=p_3$ \\\hline 
        $132$ & $121$ & $(2,1)$\\
        $213$ & $113$ & $(3,3)$ \\
        $231$ & $121$ & $(1,1)$ \\
        $312$ & $112$ & $(2,2)$ \\
        $321$ & $111$ & $(1,1)$
    \end{tabular}
    \label{tab:ex1suffix}
\end{table}
Hence, $\Bn{3}{1}{123}=\{(1,1), (2,1), (2,2), (3,3)\}$.
\end{example}

It turns out that we can track down the last two elements of the parking permutation and parking table provided that the last two elements of both are known. By tracking down the last two elements of the parking permutation and parking table, we can obtain a recurrence that counts the number of parking functions whose parking permutations avoid the pattern $123$.

\begin{lemma}\label{lem:lengthtwopairssuffixes}
Given $(i,r)\in \Bn{n}{1}{123}$, the set of pairs $(j,s)$ for which
\[((j,i),(s,r)) \in \Bn{n}{2}{123;\, (i,r)}\]
is characterized as follows:
\[
\{(j,s): \, ((j,i),(s,r))\in \Bn{n}{2}{123;\, (i,r)}\}=\begin{cases}
   \{(j,s): \,  1\leq s <j, i< j <n   \} \cup \{(n,n-1)\} & \text{if } r=1\\
   \{(1,1)\} & \text{otherwise.}
\end{cases}
\]
\end{lemma}
\begin{proof}
We first prove that if $(i,r)\in \Bn{n}{1}{123}$ then $(\tau, p')\in \Bn{n}{2}{123;\, (i,r) }$ is of the desired form. 

Let $\pi=\pi_1\pi_2\cdots \pi_{n-2}ji\in \Av_n(123)$ and $p=(p_1,p_2,\ldots, p_{n-2},s,r)$ be its parking table.
By the suffix scheme from Section~\ref{subsec:123perm}, if $\pi$ ends with $ji$ then either $j>i$ or $j<i$. 

If $j>i$ then $\ell(\pi_n=i)=1$ and so $r=1$ in the parking table $p$. We have the following cases:
\begin{enumerate}
    \item Suppose that $j=n$. Then, $\pi=\pi_1\pi_2\cdots \pi_{n-2}ni$ and $s=\ell(\pi_{n-1}=n)=n-1$ since for all $1\leq k \leq n-2$, $\pi_k<n$. Thus, $\{( (n,i),(n-1,1))\}\in \Bn{n}{2}{123; \, (i,r)}$ whenever $r=1$ and $\pi_{n-1}=n,\pi_n=i$.
    \item If $j<n$ then $i+1\leq j\leq n-1$ as $j>i$. We show that $1\leq s=\ell(\pi_{n-1}=j)\leq j-1$. If $s=\ell(\pi_{n-1}=j)$ then the only possible elements $S$ contiguously to the left of $\pi_{n-1}=j$ such that $1\leq s \leq j-1$ are if $S=\{1,2,\ldots, i-1\}\cup \{i+1,i+2,\ldots, j-1\}$, where $|S|=j-2$. Hence, $1\leq s \leq j-2+1=j-1$. It follows that $\{(\tau, p'): \,  1\leq s \leq j-1, i+1\leq j \leq n-1   \} \in \Bn{n}{2}{123;\, (i,r)}$ whenever $r=1$ and $\pi_{n-1}<n, \pi_n=i$. 
\end{enumerate}

If $j<i$ then $r=\ell(\pi_n=i)>1$. But $j=1$ by the suffix scheme from Section~\ref{subsec:123perm}. Therefore, $\pi_{n-2}>\pi_{n-1}=1$. It follows that $s=\ell(\pi_{n-1}=1)=1$. Thus, $\{((1,i), (1, r)\} \in \Bn{n}{2}{123;\, (i,r)}$ whenever $\pi_{n-1}=j,\pi_n=i$ with $j<i$ and $p_n=r>1$.

Now, we construct a permutation $\pi$ such that its bi-suffix is an element of $\Bn{n}{2}{123;\, (i,r)}$ and $\pi\in \Av_n(123)$. If $r=1$ then construct $\pi$ such that 
\[ \pi=\underbrace{(n-1)(n-2)\cdots (i+1)}_{n-i-1}\underbrace{(i-1)(i-2)\cdots 321}_{i-1}ni \]
or
\[ \pi=\underbrace{n(n-1)\cdots (j+2)(j+1)}_{n-j}\underbrace{(j-1)(j-2)\cdots (i+1)}_{j-i-1}\underbrace{(i-1)(i-2)\cdots 321}_{i-1}ji\]
with $j>i$. Both permutations are $123$-avoiding and both satisfy the condition that if $\pi_{n-1}=j$ then either $j=n$ or $i<j<n$ with $s=\ell(\pi_{n-1})=n-1$ or $1\leq s=\ell(\pi_{n-1}=j)\leq j-2+1=j-1$, respectively.
\end{proof}

\begin{example}
For $n=4$, there are fourteen $123$-avoiding permutations, and there are seven bi-suffixes where the permutation and parking table components are of length $1$. For each of those bi-suffixes, Table~\ref{tab:bisuffixeslengthtwo} lists the bi-suffixes where the permutation and parking table components are of length $2$.

\begin{table}[t]
    \centering
    \begin{tabular}{c|c|c|c}
       Bi-suffix $(i,r)$ & Bi-suffix $((j,i), (s,r)) $ & Permutation $\pi$ & Parking table $p$ \\\hline 
       $(1,1)$& $((2,1),(1,1))$ & $3421$ & $1211$  \\
       & & $4321$ & $1111$ \\ \cline{2-4}
       & $((3,1),(1,1))$ & $2431$ & $1211$ \\ \cline{2-4}
       & $((3,1), (2,1))$ & $4231$ & $1121$ \\ \cline{2-4}
       & $((4,1),(3,1))$ & $3241$ & $1131$ \\ \hline
       $(2,1)$ & $((3, 2), (1, 1))$ & $1432$ & $1211$ \\ \cline{2-4} 
       & $((3, 2), (2, 1))$ & $4132$ & $1121$ \\ \cline{2-4} 
       & $((4,2), (3,1)) $ & $3142$ & $1131$ \\ \hline 
       $(2,2)$ & $((1,2), (1,2))$ & $3412$ & $1212$ \\
       & & $4312$ & $1112$ \\ \hline 
       $(3,1)$  & $((4,3), (3,1))$ & $2143$ & $1131$ \\ \hline 
       $(3,2)$  &$((1,3), (1,2))$ & $2413$ & $1212$ \\ \hline 
       $(3,3)$  & $((1,3), ((1,3)) $ & $4213$ & $1113$ \\ \hline
       $(4,4)$ & $((1,4), (1,4))$ & $3214$ & $1114$ 
    \end{tabular}
    \caption{The bi-suffixes where the permutation and parking table components are of length $2$.}
    \label{tab:bisuffixeslengthtwo}
\end{table}
\end{example}

\begin{proposition}
Let $a(n)$ be the number of parking functions whose parking permutations avoid $123$. For each bi-suffix
$S=(i,r)\in \Bn{n}{1}{123}$, let $a(n,S)$ denote the number of parking functions whose parking permutations $\pi$ have bi-suffix $S$. Then
\[
a(n)=\sum_{S\in \Bn{n}{1}{123}} a(n,S),
\]
with initial conditions $a(1,(1,1))=1$ and $a(1,S)=0$ otherwise.

Fix $S=(i,r)\in \Bn{n}{1}{123}$. For each pair $(j,s)$ such that
\[
((j,i),(s,r))\in \Bn{n}{2}{123;\, (i,r)},
\]
let $S'=((j,i),(s,r))$ and $a(n,S,S')$ denote the number of parking functions whose parking permutations $\pi$ have bi-suffix $S'$.
Then
\[
a(n,S)
=
\sum_{\substack{(j,s)\text{ such that}\\
((j,i),(s,r))\in
\Bn{n}{2}{123;\, (i,r)}}}
a(n,S,S').
\]
Furthermore,
\[a(n,S,S')=a(n,(i,r),(j,s))=
\begin{cases}
a(n-1,(j-1,s)),
& \text{if } r=1 \text{ and } j>i,\\
\dfrac{r}{r-1}\,a(n-1,(i-1,r-1)),
& \text{if } r>1 \text{ and } j=s=1.
\end{cases}
\]
\end{proposition}

\begin{proof}
The proof involves deleting exactly the elements given by the suffix permutation scheme and reducing the permutation to a smaller permutation. For each permutation, the recurrence admits the suffixes that are characterized in Lemma~\ref{lem:lengthtwopairssuffixes}. We omit the proof as it follows from the previous section.
\end{proof}

\section{Main Result}\label{sec:mainresult}

In the previous section, we described the suffix enumeration scheme that is extendable to pattern avoidance in parking functions. We now present the recurrence that provides the number of parking functions whose parking permutations avoid the patterns $1234$ and $2134$. We begin by describing the enumeration scheme for the $\{1234,2134\}$-avoiding permutations and then use it for the parking functions setting.

\subsection{The permutation scheme}\label{subsec:12342134scheme}
The enumeration scheme for the $\{1234, 2134\}$-avoiding permutations is similar to the scheme for the $123$-avoiding permutations, described in Section~\ref{subsec:123perm}. 
We describe it here but leave out some of the details.

Let $A(n)=\Av_n(1234,2134)$ 
and define 
\[A_1(n,i)\coloneq \{\pi \in A(n): \, \pi_n=i\}. \]
Hence,
\[A(n)=\bigcup_{i=1}^n A_1(n,i).\]

The enumeration scheme for the permutations in the set $A(n)$ is depth $2$. For $n\geq 2$, define
\[A_{12}(n,j,i)\coloneq \{\pi\in A(n): \, \pi_{n-1}=j, \pi_{n}=i, j<i\}, \]
and
\[A_{21}(n,h,i)\coloneq \{\pi\in A(n): \, \pi_{n-1}=h, \pi_n=i, h>i\}. \]

As in the $123$-avoiding permutations case, if $i=1$ or $i=2$ then $A_1(n,i)$ is equinumerous to the set $A(n-1)$. If $i\geq 3$, then we have
\[A_1(n,i)=\left(\bigcup_{j=1}^{i-1} A_{12}(n,j,i)\right)\cup \left(\bigcup_{h=i+1}^{n} A_{21}(n,h,i)\right).\]

Following similar arguments discussed in Section~\ref{subsec:123perm}, we obtain $|A_{21}(n,h,i)|=|A_1(n-1,h-1)|$.
For the set $A_{12}(n,j,i)$, the scheme relies on gap vectors. It turns out that $A_{12}(n,j,i)\neq \emptyset$ if and only if $j=1$ or $j=2$. For any $j>2$, $|A_{12}(n,j,i)|=0$. Hence, if $j=1$ then $|A_{12}(n,j,i)=A_1(n-1,i-1)$. Similarly, if $j=2$ then $|A_{12}(n,j,i)=A_1(n-1,i-1)$. 

Therefore, if $i\geq 3$, then
\[|A_1(n,i)|=\sum_{h=1}^{n-1} |A_1(n-1,h)| + 2|A_1(n-1,i-1)|, \]
otherwise if $i=1$ or $i=2$ then $|A_1(n,i)|=|A(n-1)|$.

\begin{proposition}
Let $a(n)$ be the number of $1234,2134$-avoiding permutations and $a(n,i)$ be those permutations that end with the element $i$. Then
\[a(n,i)=
\begin{cases}
    a(n-1) & \text{if } i=1 \text{ or } i=2 \\
    2a(n-1,i-1)+\sum_{j=i}^{n-1} a(n-1,j) & \text{otherwise}
\end{cases}
\]
with initial conditions $a(1,1)=1$ and $a(1,i)=0$ if $i\geq 2$.
\end{proposition}

\subsection{The parking enumeration scheme}
We begin with a claim that is used in the proof of the main result. Recall that if $\pi\in \Av_n(1234,2134)$ and $\pi_{n-1}< \pi_n$ then $\pi_{n-1}=1$ or $\pi_{n-1}=2$.

\begin{lemma}\label{lemm:item1-r1}
Let $\pi\in \Av_n(1234,2134)$ and $\pi_{n-1}< \pi_n$. 
Let $ k \in [n-2]$ and $\pi_k=1$. If $\pi_{n-1}=2$ and $\pi_n=n$ then
\[ (1, \pi_{k+1},\pi_{k+2}, \ldots, \pi_{n-2},2,n)=(1,\underbrace{r_1,r_1-1\ldots, 3}_{r_1-2},2,n),\]
whenever $(1,\pi_{k+1},\pi_{k+2},\ldots,\pi_{n-2},2)$ is of length $r_1$, $2\leq r_1\leq n-1$. As a consequence, the parking table of $(1,\underbrace{r_1,r_1-1\ldots, 3}_{r_1-2},2)$ is $(1,2,\underbrace{1,\ldots, 1}_{r_1-2})$.
\end{lemma}
\begin{proof}
Let $(1,\pi_{k+1},\pi_{k+2},\ldots,\pi_{n-2},2)$ be of length $r_1$, $0\leq r_1\leq n-1$. Assume to the contrary that $\pi_a>r_1$ for some $a\in \NN$ where $k<a\leq n-2$. Then there exists $b\leq k$ such that $\pi_b<r_1$ and $\pi_b$ is to the left of $\pi_k=1$. It follows that $\pi_b, \pi_k=1, \pi_a, n$ is order isomorphic to $2134$ since $1=\pi_k<\pi_b<r_1<\pi_a<n$, a contradiction. Thus, $\pi_{k+1}>\pi_{k+2}>\cdots>\pi_{n-2}$ and $\pi_{k+1}=r_1, \pi_{k+2}=r_1-2,\ldots, \pi_{n-2}=2$.
\end{proof}

\begin{lemma}\label{lemm:item2-decreasing}
Let $\pi\in \Av_n(1234,2134)$ and $\pi_{n-1}< \pi_n$. 
Let $i\in[n], r\in[i-1]$, and $r<i$. Let $\pi_n=i, \pi_{n-1}=2$ and $\pi_a=1$ for some $a\leq n-r$ such that $\pi_k<\pi_n=i$ for all $n-r+1\leq k \leq n-2$ and $\pi_{n-r}>\pi_n=i$. Then 
    \[ (\pi_{n-r+1},\pi_{n-r+2},\ldots, \pi_{n-2}, 2,i)= (r,r-1,\ldots, 3,2,i).\]
\end{lemma}
\begin{proof}
Let $\pi_a=1$ for some $a\leq n-r$, and $\pi_n=i, \pi_{n-1}=2$ such that $\pi_k<\pi_n=i$ for all $n-r+1\leq k \leq n-2$. By assumption, there are $r-2$ positions between $\pi_{n-r}$ and $\pi_{n-1}=2$ and the elements at those positions must be from the set $\{3,4\ldots, r, r+1,\ldots, i-1\}$. Suppose that there exists at least one $b>n-r$ such that $i>\pi_b>r$. This implies that there is at least one $c\in[n-r]$ such that $\pi_c<r$ as there are $r-2$ elements in the set $\{3,4,\ldots, r\}$ and only $r-3$ possible positions to the right of $\pi_{n-r}$.
There are two cases, either $\pi_c$ is somewhere to the left of $\pi_a=1$ or $\pi_c$ is somewhere to the right of $\pi_a=1$ (and before $\pi_{n-r+1}$). If $\pi_c$ is to the left of $\pi_a=1$ then the subsequence $\pi_c<r, \pi_a=1, \pi_b>r, \pi_n=i$ is order isomorphic to $2134$ as $\pi_a=1<\pi_c<r< \pi_b <\pi_n=i$. Otherwise, $\pi_c$ is to the right of $\pi_a=1$ then the subsequence $\pi_a=1, \pi_c<r, \pi_b>r, \pi_n=i$ is order isomorphic to $1234$. In either case, both lead to a contradiction as $\pi$ is $\{1234,2134\}$-avoiding. In addition, the elements $\pi_{n-r+1},\pi_{n-r+2},\ldots, \pi_{n-2}$ must be in decreasing order because if $\pi_d<\pi_e$ for $n-r+1\leq d<e\leq n-2$, then the subsequence $\pi_a=1, \pi_a,\pi_b,\pi_n=i$ forms a $1234$ pattern. Thus, the result follows.
\end{proof}

The following definition and claim is necessary for the main result. 

\begin{definition}
For a collection of permutations $\sigma_1,\ldots,\sigma_k$, let $\Pk_{n}(\sigma_1,\ldots,\sigma_k)$ be the set of parking functions $\alpha$ such that its associated parking permutation $\pi$ contains none of $\sigma_1,\ldots,\sigma_k$ as a pattern and let $\pk_{n}(\sigma_1,\ldots,\sigma_k)=|\Pk_{n}(\sigma_1,\ldots,\sigma_k)|$. 
\end{definition}

For brevity, we let $\pk(n)=\pk_{n}(1234,2134)$ and for each bi-suffix $(i,r)\in \Bn{n}{1}{1234,2134}$:
\begin{enumerate}
    \item $\pk(n,i)$ denotes the number of parking functions whose parking permutations $\pi$ satisfy $\pi_n=i$,
    \item $\pk(n,i,r)$ denotes the number of parking functions whose parking permutations $\pi$ have bi-suffix $(i,r)$.
\end{enumerate}

\begin{claim}
For $1\leq i \leq n$,
\[ \pk(n,i)=\sum_{r=1}^i \pk(n,i,r).\]
Take $\pi \in \Av_n(1234,2134)$ and suppose that $\pi_n=i$. If 
\[r=\ell(\pi_n=i)=\max\{r: \, \pi_j \leq \pi_n=i \text{ for all } n-r+1\leq j\leq n\},\] then $1\leq r \leq i$, as there are $i-1$ possible elements contiguously to the left of $\pi_n=i$ smaller than $\pi_n=i$.
\end{claim}

\begin{theorem}
Let $\pk(n)$ be the number of parking functions whose parking permutations avoid the patterns $1234$ and $2134$. Then
\[\pk(n)=\sum_{i=1}^n\pk(n,i),\]
with $\pk(0)=1$.
For each bi-suffix $S=(i,r)\in \Bn{n}{1}{1234,2134}$, let $\pk(n,i,r)$ denote the number of parking functions whose parking permutations $\pi$ have bi-suffix $S$. Then
\[\pk(n,i)=\sum_{r=1}^i \pk(n,i,r),\]
and 
\[ \pk(n,i,r)=
\begin{cases}
\sum_{j=i}^{n-1} \pk(n-1,j) & \text{if } r=1\\
2\pk(n-2) & \text{if } r=2 \text{ and } i=2\\
4\pk(n-1,i-1,1) & \text{if } r=2 \text{ and } i\neq 2\\
\pk(n-r)\pk(r,r) & \text{if } i=r \text{ and } i<n \\
\frac{n}{n-1}\pk(n-1,n-1,n-1)+\sum_{r_1=2}^{n-1} \frac{2n}{n-r_1}\pk(n-r_1,n-r_1,n-r_1) & \text{if } i=r \text{ and } i=n\\
\frac{r}{r-1}\pk(n-1,i-1,r-1) + r\pk(n-r+1,i-r+1,1) & \text{otherwise} \\ 
\qquad  + \sum_{r_1=2}^{r-1} \frac{2r}{r-r_1}\pk(n-r_1,i-r_1,r-r_1) 
\end{cases}
\]
with 
\[\pk(1,i,r)=
\begin{cases}
    1, & (i,r)=(1,1) \\
    0, & \text{otherwise,}
\end{cases} 
\quad 
\pk(2,i,r)=
\begin{cases}
    1, & (i,r)=(1,1) \\
    2, & (i,r)=(2,2) \\
    0, & \text{otherwise,}
\end{cases}
\]
and
\[
\pk(3,i,r)=
\begin{cases}
    3, & (i,r)=(1,1) \\
    2, & (i,r)=(2,1) \text{ or } (i,r)=(2,2)\\
    9, & (i,r)=(3,3) \\
    0, & \text{otherwise.}
\end{cases}
\]
The first terms of the sequence are
\[ 1, 3, 16, 89, 534, 3394, 22447, 152901, 1065508, 7560824, \ldots \]
\end{theorem}
The strategy of the proof is as follows. We begin with a permutation $\pi$ of length $n$ and parking table $p$. By the permutation scheme, we reduce the permutation $\pi$ by deleting elements according to the scheme. This deletion causes some elements of the parking table $p$ to change (depending on the deletion). Since the product of the elements of $p$ yields the number of parking functions that lead to the parking permutation $\pi$, to recover this number, we multiply the new product of the elements in the new parking table by some factor that depends on the deletion.
\begin{remark}
Whenever a word $w$ of length $n$ is reduced to a permutation $\red(w)$ of length $m$, we simply write $\red(w)=w_1w_2\cdots w_m$ where for each $k$, $w_k$ is the corresponding element after reduction. For the main proof, we only keep track of the last two elements depending on deletion.
\end{remark}
\begin{proof}
Let $\Pi=\{1234,2134\}$ and $\pi\in \Av_n(\Pi)$ such that $\pi_{n-1}=j,\pi_n=i$ and $p=(p_1,p_2,\ldots, p_n)$ is the parking table of $\pi$ such that and $p_n=r$. 
We prove the recurrence via the following cases on the value of $r$:
\begin{enumerate}
    \item If $r=1$: In this case, $\pi_{n-1}>\pi_{n}$ and so $\pi_{n-1},\pi_n$ reduce to $2,1$. Hence, $i+1\leq j \leq n$. By the permutation scheme, delete $\pi_n=i$ from $\pi$ and apply the reduction map to obtain a permutation of length $n-1$. The resulting word after deletion is $w=\pi_1\pi_2\cdots \pi_{n-2}j$ and reducing $w$ to a permutation of length $n-1$ results in a permutation whose last element is $j-1$. By applying Lemma~\ref{ex:suffix1} to the reduced permutation, we obtain
    \[  \pk(n,i,r) = \sum_{j=i}^{n-1} \pk(n-1,j).\]
    \item If $r=2$: In this case, if $r=2$ then $j=\pi_{n-1}<\pi_n=i$ and $\pi_n<\pi_{n-2}$. Hence, we proceed with the cases $\pi_n=2$ or $\pi_n\neq 2$ from the permutation scheme.
    \begin{enumerate}
        \item If $\pi_n=2$ then $\pi_{n-1}=1$ and the parking table is $p=(p_1,p_2,\ldots, 1,2)$. By the permutation scheme, delete $\pi_{n-1}=1$ from $\pi$ and apply the reduction map to obtain a permutation of length $n-1$. The resulting word after deletion is $w=\pi_1\pi_2\cdots \pi_{n-2}2$ and reducing $w$ to a permutation of length $n-1$ results in $\red(w)=w_1w_2\cdots w_{n-2}1$. The parking table of $w$ and $\red(w)$ coincide for the first $n-2$ elements by Claim~\ref{lem:parkingtables}. The corresponding parking table of $\red(w)$ is $p'=(p_1,p_2,\ldots, p_{n-2},1)$ since the last element of $\red(w)$ equals to one. Furthermore, since $\red(w)$ ends with the element one, deleting it and reducing again yields a permutation of length $n-2$ and its parking table is $p''=(p_1,p_2,\ldots, p_{n-2})$. 
        In particular, to recover the product $p_1p_2\cdots p_{n-2}12$, we multiply the product $p_1p_2\cdots p_{n-2}$ by $2$, which results in the number of parking functions whose parking permutations avoid $\Pi$. Thus, $\pk(n,i,r)=2\pk(n-2)$.
        \item If $\pi_n\neq 2$ with $j<i$ and $r=2$ then $j=1$ or $j=2$ by the permutation scheme. Since $r=2$ then $\pi_{n-2}>\pi_n=i$.\label{case2b}
        \begin{enumerate}
            \item If $j=1$ then $\pi=\pi_1\pi_2\cdots 1i$ ($i>2$) and the parking table is $p=(p_1,p_2,\ldots, 1,2)$. By the permutation scheme, delete $\pi_{n-1}=1$ and apply the reduction map to obtain a permutation of length $n-1$. The resulting word after deletion is $w=\pi_1\pi_2\cdots \pi_{n-2}i$ and reducing $w$ to a permutation of length $n-1$ results in a permutation whose last element is $i-1$ (since $j=1<i$). Thus, $\red(w)=w_1w_2\cdots w_{n-2}(i-1)$. The corresponding parking table of $\red(w)$ is $p'=(p_1,p_2,\ldots, p_{n-2},1)$ since $w_{n-2}\geq i>i-1$ as $\pi_{n-2}>\pi_n=i$. Moreover, to recover the product $p_1p_2\cdots p_{n-2}12$, we multiply the product $p_1p_2\cdots p_{n-2}$ by $2$, which results in the number $\pk(n,i,r)$. Thus, $\pk(n,i,r)=2\pk(n-1,i-1,1)$.
            \item If $j=2$ then $\pi=\pi_1\pi_2\cdots 2i$ ($i>2$) and the parking table is $p=(p_1,p_2,\ldots, 1,2)$ (since $\pi_{n-2}>i>2$). Then again, delete $\pi_{n-1}=2$ and a similar argument as in the case for $j=1$ shows that $\pk(n,i\neq 2,r)=2\pk(n-1,i-1,1)$.
        \end{enumerate}
        
        Altogether, for case~(\ref{case2b}), we obtain $\pk(n,i\neq 2,r)=2\pk(n-1,i-1,1)+2\pk(n-1,i-1,1)=4\pk(n-1,i-1,1)$.
    \end{enumerate}
    
    \noindent The next two cases are broken up into the cases $r=i<n$ or $r=i=n$.
    
    \item If $2<r<n$ and $r=i$: In this case, $r>2$ implies that $j=\pi_{n-1}<\pi_n=i$. By the permutation scheme $j=1$ or $j=2$. Moreover, $\pi$ can be decomposed in the following form:
    \[\pi=\underbrace{\pi_1\pi_2\cdots \pi_{n-i}}_{n-i}\underbrace{\pi_{n-i+1}\pi_{n-i+2}\cdots \pi_{n-2}j}_{i-1} i\]
    where $\pi_k<\pi_n=i$ for any $k, n-i<k<n$, and $\pi_{n-i}>\pi_n=i$. In particular, if $j=1$ then \[\{\pi_{n-i+1},\pi_{n-i+2},\ldots ,\pi_{n-2}\}=\{2,3,\ldots, i-2,i-1\},\]
    and if $j=2$ then
    \[ \{\pi_{n-i+1},\pi_{n-i+2},\ldots ,\pi_{n-2}\}=\{1,3,\ldots, i-2,i-1\}.\]
    This is true because every element in the set $\{\pi_{n-i+1},\pi_{n-i+2},\ldots ,\pi_{n-2},j\}$ is smaller than $i$ and there are $i-1$ elements in the set. Thus, the number of parking functions whose parking permutations avoid the patterns $\Pi$ is equal to $\pk(n,i,r)=\pk(n-i)\pk(i,i)$ with $i=r$. Hence, $\pk(n,i,r)=\pk(n-r)\pk(r,r)$. \label{case3}
    \item If $r=n$ and $r=i$: In this case, $r=n$ implies that $j=\pi_{n-1}<\pi_n=n$. By the permutation scheme $j=1$ or $j=2$. The key to this case is noting that decomposing $\pi$ into two block as in case~(\ref{case3}) is not possible. Indeed, all elements from $\pi_1$ to $\pi_{n-1}$ are less than $\pi_n=n$. Hence, we proceed with the cases $\pi_{n-1}=1$ or $\pi_{n-1}=2$ from the permutation scheme.\label{item:case4}
    \begin{enumerate}
        \item If $\pi_{n-1}=1$ then $\pi=\pi_1\pi_2\cdots 1n$ and the parking table is $p=(p_1,p_2,\ldots, 1,n)$. By the permutation scheme, delete $\pi_{n-1}=1$ and apply the reduction map to obtain a permutation of length $n-1$. The resulting word after deletion is $w=\pi_1\pi_2\cdots \pi_{n-2}n$ and reducing $w$ to a permutation of length $n-1$ results in $\red(w)=(\pi_1-1)(\pi_2-1)\cdots (\pi_{n-2}-1) (n-1)$, since for every $k\in[n-2]$, we have $\pi_k>\pi_{n-1}=1$. The corresponding parking table of $\red(w)$ is $p'=(p_1,p_2,\ldots, p_{n-2}, n-1)$, where the first $n-2$ elements coincide with the parking table of $p$ by Claim~\ref{lem:parkingtables}. Therefore, to recover the product $p_1p_2\cdots p_n$, we multiply $p_1p_2\cdots p_{n-2}(n-1)$ by $\frac{n}{n-1}$. Hence $\pk(n,i,r)=\frac{n}{n-1}\pk(n-1,n-1,n-1)$.
        \item If $\pi_{n-1}=2$ then $\pi=\pi_1\pi_2\cdots 2 n$ and the parking table is $p=(p_1,p_2,\ldots, p_{n-1},n)$. Then there exists an index $k$ such that $\pi_k=1$ for some $1\leq k \leq n-2$. Let $k$ be such index and consider the subsequence $w=1\pi_{k+1}\pi_{k+2}\cdots 2$ of $\pi$. By Lemma~\ref{lemm:item1-r1} together with $\pi_{k+1}>\pi_k=1$, we obtain that the parking table of $w$ is $p'=(1,2,1,\ldots, 1)$. Suppose that $w=1\pi_{k+1}\pi_{k+2}\cdots 2$ is of length $r_1$, where $2\leq r_1\leq n-1$. The bounds on $r_1$ follows from observing that $w$ can be at most length $n-1$ if $\pi_1=1$ and it can be at least length $2$ if $\pi_{n-2}=1$. Hence, the parking table $p$ of $\pi$ is equal to \[p=(p_1,p_2,\ldots, p_{r_1-1},1,2,\underbrace{1,\ldots, 1}_{r_1-2}, n).\]
        Delete $w=1\pi_{k+1}\pi_{k+2}\cdots 2$ from $\pi$ and apply the reduction map to obtain a permutation of length $n-r_1$. The resulting word after deletion is $w'=\pi_1\pi_2\cdots \pi_{n-r_1-1}n$ and reducing $w'$ to a permutation of length $n-r_1$ results in $\red(w')=w'_1w'_2\cdots(n-r_1)$. The corresponding parking table of $\red(w')$ is $q=(p_1,p_2,\ldots, p_{n-r_1-1},n-r_1)$. This follows by the deletion of exactly $r_1$ elements all less than $\pi_n=n$. Therefore, to recover the product \[p_1p_2\cdots p_{n-r_1-1}12\underbrace{1\cdots 1}_{r_1-2} n=2np_1p_2\cdots p_{n-r_1-1},\] we multiply $p_1p_2\cdots p_{n-r_1-1}(n-r_1)$ by $\frac{2n}{n-r_1}$ for all $2\leq r_1 \leq n-1$. Hence $\pk(n,i,r)=\sum_{r_1=2}^{n-1} \frac{2n}{n-r_1}\pk(n-r_1,n-r_1,n-r_1)$.
    \end{enumerate}
    Altogether, for case~(\ref{item:case4}), we obtain
    \[ \pk(n,i,r)=\frac{n}{n-1}\pk(n-1,n-1,n-1)+ \sum_{r_1=2}^{n-1} \frac{2n}{n-r_1}\pk(n-r_1,n-r_1,n-r_1).\]
    \item The last case is if $r\neq i$. In this case, if $r<i$ then $i<n$ because if not, $i=n$ implies $r=i$, which is the previous case. Thus $2<r<i$ and $j=\pi_{n-1}<\pi_n=i$. Hence, we proceed with the two cases $\pi_{n-1}=1$ or $\pi_{n-1}=2$ from the permutation scheme. \label{item:case5}
    \begin{enumerate}
        \item If $\pi_{n-1}=1$ then $\pi=\pi_1\pi_2\cdots 1i$ and the parking table is $p=(p_1,p_2,\ldots, 1,r)$. By the permutation scheme, delete $\pi_{n-1}=1$ and apply the reduction map to obtain a permutation of length $n-1$. The resulting word after deletion is $w=\pi_1\pi_2\cdots \pi_{n-2}i$ and reducing $w$ to a permutation of length $n-1$ results in $\red(w)=w_1w_2\cdots w_{n-2}(i-1)$. The corresponding parking table of $\red(w)$ is $p'=(p_1,p_2,\ldots, p_{n-2}, r-1)$. Hence, to recover the product $p_1p_2\cdots p_{n-2}1r=rp_1p_2\cdots p_{n-2}$, we multiply $p_1p_2\cdots p_{n-2}(r-1)$ by $\frac{r}{r-1}$. It follows that $\pk(n,i,r)=\frac{r}{r-1}\pk(n-1,i-1,r-1)$.
        \item If $\pi_{n-1}=2$ then $\pi=\pi_1\pi_2\cdots 2i$ and the parking table is $p=(p_1,p_2,\ldots, p_{n-1}, r)$. 
        This implies that $\pi_n=i>\pi_k$ for every $n-r+1\leq k \leq n-1$ and $\pi_{n-r}>\pi_n=i$. Hence, either $\pi_k=1$ for some $1\leq k \leq n-r-1$ or $\pi_k=1$ for some $n-r+1\leq k\leq n-2$.
        Observe that $\pi_{n-r}\neq 1$ as this would imply that $1=\pi_{n-r}<\pi_n=i$, which is not the case. We proceed with the two cases based on the position of the element $1$. 
        \begin{enumerate}
            \item If $\pi_k=1$ for some $k\in \NN$, $1\leq k \leq n-r-1$ then since $r>2$, $\pi_a>\pi_b$ for all $a,b$ such that $n-r+1\leq a<b \leq n-2$. Otherwise if $\pi_a<\pi_b$ then $\pi_{k}=1, \pi_a, \pi_b, \pi_n=i$ is order isomorphic to $1234$, a contradiction. Hence, $\pi_a>\pi_b$ for all $a,b$ with $n-r+1\leq a<b \leq n-2$. It follows that $\pi_{n-r+1}>\pi_{n-r+2}>\cdots > \pi_{n-2}$. Thus, the parking table of the word
            \[ (\pi_{n-r+1}, \pi_{n-r+2}, \ldots, \pi_{n-2},2)\]
            is equal to
            \[(1,1,\ldots, 1,1),\]
            with $r-1$ ones.
            Therefore, the parking table of \[\pi=\pi_1\pi_2\cdots \pi_{n-r}\pi_{n-r+1}\cdots \pi_{n-2}2i\] is \[p=(p_1,p_2,\cdots, p_k=1, p_{k+1}, \ldots, p_{n-r}, \underbrace{1,1,\ldots, 1}_{r-1}, r).\]
            Delete $\pi_{n-r+1}\pi_{n-r+2}\cdots \pi_{n-2}2$ from $\pi$ and apply the reduction map to obtain a permutation of length $n-r+1$. The resulting word after deletion is $w=\pi_1\pi_2\cdots \pi_{n-r} i$ and reducing $w$ to a permutation of length $n-r+1$ results in $\red(w)=w_1w_2\cdots (i-r+1)$. This follows by the deletion of exactly $r-1$ elements smaller than $\pi_n=i$. The corresponding parking table of $\red(w)$ is \[p'=(p_1,p_2,\ldots, p_{n-r}, 1),\] since $\pi_n=i<\pi_{n-r}$ and hence $w_{n-r+1}=i-r+1<w_{n-r}=\pi_{n-r}-r+1$ as exactly $r-1$ elements are deleted all less than $\pi_{n-r}$.
            Hence, to recover the product \[p_1p_2\cdots p_{n-r}\underbrace{1\cdots 1}_{r-1}r=rp_1p_2\cdots p_{n-r},\] we multiply $p_1p_2\cdots p_{n-r}$ by $r$. It follows that $\pk(n,i,r)=r \pk(n-r+1, i-r+1,1)$.
            \item If $\pi_k=1$ for some $k\in \NN$, $n-r+1\leq k \leq n-2$, then \[\pi=\pi_1\pi_2\cdots \pi_{n-r}\pi_{n-r+1}\cdots \pi_{k-1}1\pi_{k+1}\cdots \pi_{n-2}2i\] and the parking table is 
            \[p=(p_1,p_2,\ldots , p_{n-r}, p_{n-r+1},\ldots, p_{k-1},1, p_{k+1}, \ldots, p_{n-2},2, r).\] 
            Observe that $\pi_k=1<\pi_{k+1}$ and $\pi_{a}>\pi_b$ for all $a,b$ such that $k+1\leq a<b \leq n-1$. Otherwise, if $\pi_a<\pi_b$ then $\pi_k=1, \pi_a,\pi_b, \pi_n=i$ is order isomorphic to $1234$, a contradiction. Hence, $\pi_k=1<\pi_{k+1}$ and $\pi_{a}>\pi_b$ for all $a,b$ such that $k+1\leq a<b \leq n-1$. It follows that $\pi_{k+1}>\pi_{k+2}>\cdots > \pi_{n-2}>2$. Let $1\pi_{k+1}\pi_{k+2}\cdots \pi_{n-2}2$ be of length $r_1$ for $2\leq r_1\leq r-1$. 
            The bounds on $r_1$ follows from observing that $r_1=2$ if $\pi_{n-2}=1$, and at most equals to $r-1$ if $\pi_{n-r+1}=1$ and $\pi_{n-r}\neq 1$. Hence, the parking table of the word \[(\pi_k=1, \pi_{k+1}, \ldots , \pi_{n-2}, 2)\] is equal to
            \[(1,2,\underbrace{1,\ldots, 1}_{r_1-2}).\]
            Thus, the parking table of \[\pi=\pi_1\pi_2\cdots \pi_{n-r}\pi_{n-r+1}\cdots \pi_{k-1}1\pi_{k+1}\cdots \pi_{n-2}2i,\] is
            \[p=(p_1,p_2,\ldots, p_{n-r}, p_{n-r+1},\ldots, p_{n-r_1-2}, p_{n-r_1-1}, 1,2,\underbrace{1,\ldots, 1}_{r_1-2}, r).\]
            Delete $1\pi_{k+1}\pi_{k+2}\cdots \pi_{n-2}2$ from $\pi$ and apply the reduction map to obtain a permutation of length $n-r_1$. The resulting word after deletion is $w=\pi_1\pi_2\cdots\pi_{n-r}\pi_{n-r+1}\cdots \pi_{n-r_1-1}i$ and reducing $w$ to a permutation of length $n-r_1$ results in \[\red(w)=w_1w_2\cdots w_{n-r}r_{n-r+1}\cdots w_{n-r_1-1}(i-r_1).\] The corresponding parking table of $\red(w)$ is \[p'=(p_1,p_2,\ldots, p_{n-r}, p_{n-r+1}, \ldots, p_{n-r_1-1}, r-r_1).\] This follows by the deletion of exactly $r_1$ elements smaller than $\pi_n=i$. Thus, to recover the product
            \begin{align*}
                \qquad \qquad  p_1p_2\cdots p_{n-r}p_{n-r+1}\cdots p_{n-r_1-2}p_{n-r_1-1}12\underbrace{1\cdots 1}_{r_1-2}r=2rp_1p_2\cdots p_{n-r_1-2}p_{n-r_1-1},
            \end{align*}
            we multiply 
            $p_1p_2\cdots p_{n-r}p_{n-r+1}\cdots p_{n-r_1-1}(r-r_1)$ by $\frac{2r}{r-r_1}$.
            Thus, for all $2\leq r_1\leq r-1$, \[\pk(n,i,r)=\sum_{r_1=2}^{r-1} \frac{2r}{r-r_1}\pk(n-r_1,i-r_1,r-r_1). \]
        \end{enumerate}
    \end{enumerate}
    Altogether, for case~(\ref{item:case5}), we obtain 
    \begin{align*}
        \pk(n,i,r)&=\frac{r}{r-1}\pk(n-1,i-1,r-1) + r\pk(n-r+1,i-r+1,1) \\
        &\qquad + \sum_{r_1=2}^{r-1} \frac{2r}{r-r_1}\pk(n-r_1,i-r_1,r-r_1).
    \end{align*}
\end{enumerate}
In all of the arguments above, by Claim~\ref{claim:reduavoid}, each reduced permutation avoids both patterns $1234, 2134$. In addition, it is easy to check that the initial conditions are satisfied for $n\leq 3$.
\end{proof}

\section*{Acknowledgements}
L.~Martinez was supported by the NSF Graduate Research Fellowship Program under Grant No. 2233066 and in part by a Joel Lebowitz Summer Research Fellowship.

\bibliographystyle{plain}
\bibliography{bibliography.bib}

\end{document}